%% file: locationMax12.tex
\documentclass[12p, reqno]{amsart}
\usepackage{amsfonts}
\usepackage{amsmath}
\usepackage{amssymb}
\usepackage[pdftex]{graphicx,color}

\newtheorem{theorem}{Theorem}[section]

\theoremstyle{definition}

\theoremstyle{example}

\theoremstyle{proposition}
\newtheorem{proposition}[theorem]{Proposition}
\theoremstyle{remark}
\newtheorem{remark}[theorem]{Remark}
\theoremstyle{corollary}

\numberwithin{equation}{section}

\DeclareMathOperator{\re}{Re}
\DeclareMathOperator{\im}{Im}
\usepackage{graphicx}
\graphicspath{%
    {converted_graphics/}
    {/}
}
\begin{document}
\def\R{{\mathbb R}}
\def\C{{\mathbb C}}
\def\N{{\mathbb N}}
\def\DD{{\mathbb D}}
\def\S{{\mathbb S}}
\def\rr{{\cal R}}
\def\e{\emptyset}
\def\dQ{\partial Q}
\def\dk{\partial K}
\def\endofproof{{\rule{6pt}{6pt}}}
\def\di{\displaystyle}
\def\dist{\mbox{\rm dist}}
\def\u-{\overline{u}}
\def\du{\frac{\partial}{\partial u}}
\def\dv{\frac{\partial}{\partial v}}
\def\dt{\frac{d}{d t}}
\def\dx{\frac{\partial}{\partial x}}
\def\con{\mbox{\rm const }}
\def\Box{\spadesuit}
\def\ii{{\bf i}}
\def\curl{{\rm curl}\,}
\def\dive{{\rm div}\,}
\def\grad{{\rm grad}\,}
\def\dist{\mbox{\rm dist}}
\def\pr{\mbox{\rm pr}}
\def\pp{{\cal P}}
\def\supp{\mbox{\rm supp}}
\def\Arg{\mbox{\rm Arg}}
\def\In{\mbox{\rm Int}}
\def\Re{\mbox{\rm Re}\:}
\def\li{\mbox{\rm li}}
\def\ep{\epsilon}
\def\tr{\tilde{R}}
\def\be{\begin{equation}}
\def\ee{\end{equation}}
\def\cn{{\mathcal N}}
\def\sn{{\mathbb  S}^{n-1}}
\def\Ker {{\rm Ker}\:}
\def\el{E_{\lambda}}
\def\Rc{{\mathcal R}}
\def\Ha{H_0^{ac}}
\def\la{\langle}
\def\ra{\rangle}
\def\Ko{\Ker G_0}
\def\Kd{\Ker G_b}
\def\hc{{\mathcal H}}
\def\caH{\hc}
\def\caO{{\mathcal O}}
\def\la{\langle}
\def\ra{\rangle}
\def\sp{\sigma_{+}}
\def\pa{\partial}
\def\et{|\eta'|^2}
\def\lg{L^2(\Gamma)}
\def\h1{H^1_h(\Gamma)}
\def\nc{{\mathcal N}}
\def\ii{{\bf i}}

\title[Maxwell eigenvalues] {Eigenvalues for Maxwell's equations with dissipative boundary conditions
}

\thanks{}


\author[F. Colombini, V. Petkov] {Ferruccio Colombini, \ Vesselin Petkov$\,^\dagger$ }
\address{Dipartimento di Matematica, Universit\`a di Pisa, Italia}
\email{colombini@dm.unipi.it}
\address{Institut de Math\'ematiques de Bordeaux, 351,
Cours de la Lib\'eration, 33405  Talence, France}
\email{petkov@math.u-bordeaux.fr}
\author[J. Rauch]{\ Jeffrey Rauch}
\thanks{$^\dagger$ The author was partially supported by the ANR project Nosevol BS01019 01 }
\address{Department of Mathematics, University of Michigan, USA}
\email{rauch@umich.edu}


\date{}

\dedicatory{}


\begin{abstract}
 Let $V(t) = e^{tG_b},\: t \geq 0,$ be the semigroup
  generated by Maxwell's equations  in an exterior domain $\Omega \subset \R^3$  with dissipative boundary condition $E_{tan}- \gamma(x) (\nu \wedge B_{tan}) = 0, \gamma(x) > 0, \forall x \in \Gamma = \pa \Omega.$  We prove that if $\gamma(x)$ is nowhere equal to 1, then 
  for every $0 < \ep \ll 1$ and every $N \in \N$ the eigenvalues of $G_b$  lie in the region $\Lambda_{\ep} \cup {\mathcal R}_N,$ where
$\Lambda_{\epsilon} = \{ z \in \C:\: |\re z | \leq C_{\epsilon} (|\im z|^{\frac{1}{2} + \epsilon} + 1), \: \re z < 0\},$
${\mathcal R}_N = \{z \in \C:\: |\im z| \leq C_N (|\re z| + 1)^{-N},\: \re z < 0\}.$   
\end{abstract}

\maketitle


\section{Introduction}

Suppose that $K \subset \{ x\in \R^3: \: |x| \leq a\}$ is an open connected domain 
and $\Omega := \R^3 \setminus \bar{K}$  
is an open connected  domain with $C^{\infty}$ smooth boundary $\Gamma$.
Consider the boundary problem
\begin{equation}  \label{eq:1.1}
\begin{aligned} 
&\partial_t E = \curl B,\qquad \partial_t B = -\curl E \quad {\rm in}\quad \R_t^+ \times \Omega,
\\
&E_{tan} - \gamma(x)(\nu \wedge B_{tan}) = 0 \quad{\rm on} \quad \R_t^+ \times \Gamma,
\\
&E(0, x) = e_0(x), \qquad B(0, x) = b_0(x). 
\end{aligned}
\end{equation}
with initial data $f = (e_0, b_0) \in  (L^2(\Omega))^6 = {\mathcal H}.$
Here $\nu(x)$ denotes the unit outward normal to $\partial \Omega$ at $x \in \Gamma$ pointing into $\Omega$, 
$\la\: , \ra$ denotes the scalar product in $\C^3$,
$u_{tan} := u - \la u, \nu\ra \nu$, and $\gamma(x) \in C^{\infty}(\Gamma)$ satisfies $\gamma(x) > 0$ for all $ x \in \Gamma.$ The solution of the problem (\ref{eq:1.1}) is given by a contraction semigroup $(E, B) = V(t)f = e^{tG_b} f,\: t \geq 0$, where the generator
$G_b$ has  domain $D(G_b)$
that is the closure in the graph norm of functions $u = (v, w) \in (C_{(0)}^{\infty} (\R^3))^3 \times (C_{(0)}^{\infty} (\R^3))^3$ satisfying the boundary condition $v_{tan} - \gamma (\nu \wedge w_{tan}) = 0$ on $\Gamma.$ 

In an earlier paper \cite{CPR1} we proved that the spectrum of $G_b$ in $\Re z < 0$ 
consists of  isolated eigenvalues with finite multiplicity.
If $G_b f = \lambda f$ with $\re \lambda < 0$, the solution $u(t, x) = V(t) f = e^{\lambda t} f(x) $ of (\ref{eq:1.1}) has exponentially decreasing global energy. Such solutions are called {\bf asymptotically disappearing} and they
are invisible for  inverse scattering problems.
It was proved \cite{CPR1} that if there is at least one eigenvalue $
\lambda$ of $G_b$ with $\re \lambda < 0$, then the wave operators $W_{\pm}$ are not complete, that is ${\text Ran}\: W_{-} \not=  {\text Ran}\: W_{+}$. Hence  we cannot define the scattering operator $S$ related to the Cauchy problem for the Maxwell system and (1.1) by the product $ W_{+}^{-1}W_{-}$. 
For the perfect conductor boundary conditions for Maxwell's 
equations, the energy is conserved in time and the unperturbed and perturbed  problems are associated to unitary groups. The corresponding scattering operator $S(z): (L^2(\S^2))^2 \to (L^2(\S^2))^2$ satisfies the identity
\begin{equation} \label{eq:1.2}
S^{-1}(z)= S^*(\bar{z}),\quad z \in \C
\end{equation}
if $S(z)$ is invertible at $z$. The scattering operator $S(z)$ defined in \cite{LP} is such that $S(z)$ and $S^*(z)$ are analytic in the "physical" half plane $\{z \in \C:\im z < 0\}$ and the above relation for conservative boundary conditions implies that $S(z)$ is invertible for $\im z > 0$. For dissipative boundary conditions the relation (\ref{eq:1.2}) in general is not true and $S(z_0)$ may have a non trivial kernel for some $z_0, \im z_0 > 0.$ Lax and Phillips \cite{LP} proved that this implies that $\ii z_0$ is an eigenvalue of $G_b$. The analysis of the location of the eigenvalues of $G_b$ is important for the location of the points where  the kernel of $S(z)$ is not trivial.\\

 The main result  of this paper is  the following (see Figure 1)
\begin{theorem} Assume that for all $x\in \Gamma$, $\gamma(x) \neq 1$.
Then for every $0 < \ep \ll 1$ and every $N \in \N$ there 
are constants $C_{\ep} > 0$ and $C_N > 0$ such that the eigenvalues of $G_b$ lie in the region $\Lambda_{\ep} \cup {\mathcal R}_N$, where
$$\Lambda_{\ep} = \{ z \in \C: \: |\re z | \leq C_{\ep} (|\im z|^{1/2 + \epsilon} + 1),\: \re z < 0\},$$
$${\mathcal R}_N = \{ z \in \C: \: |\im z| \leq C_N ( |\re z| + 1)^{-N},\: \re z < 0\}.$$
\end{theorem}

\begin{figure}[tbp] 
  \centering
  \includegraphics[bb=0 0 487 402,width=3.5in,height=2.5in,keepaspectratio]{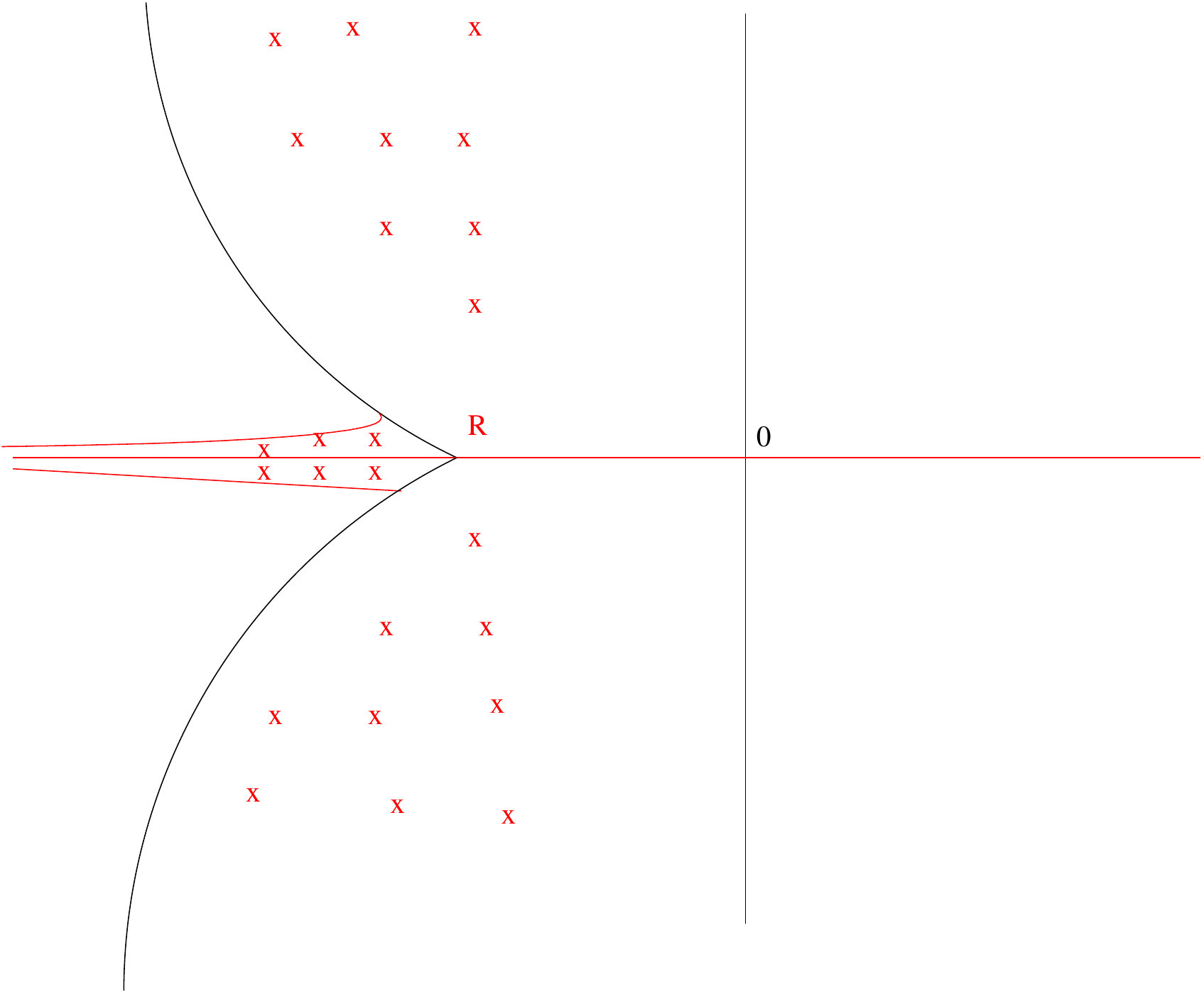}
  \caption{Eigenvalues of $G_b$}
  \label{fig:values5}
\end{figure}

If $\re \lambda < 0$ and $G_b(E, B) =\lambda (E, B)\ne 0$, then 
\begin{equation} 
\begin{aligned} 
&\lambda E = \curl B &{\rm on} &\quad \Omega,
\\
&\lambda B = - \curl E & {\rm on}& \quad \Omega,\\
&\dive E = \dive B = 0,& {\rm on} &\quad \Omega,
\\
&E_{tan} - \gamma(\nu \wedge B_{tan}) = 0 & {\rm on} & \quad \Gamma. 
\end{aligned} 
\label{eq:1.3}
\end{equation}
This implies that $u := (E, B)$ satisfies
$$ 
 \Delta u - \lambda^2 u = 0,\qquad {\rm on}\quad \Omega.
 $$

The eigenvalues of $G_b$ are symmetric with respect to the real axis, so it is sufficient to examine the location of the eigenvalues whose imaginary part is  nonnegative.  
The mapping $z\mapsto z^2$ maps the positive quadrant $\{z \in \C:\re z>0\,,\, \im z>0\}$ bijectively to
the upper half space.  Denote by $\sqrt z$ the inverse map.
The part of the spectral domain
$\{\lambda \in \C:\, \re \lambda<0\,,\, \im \lambda >0\}$
is mapped by $\lambda =\ii \sqrt z$ to the upper half plane $\{z\in \C:\, \im z>0\}$.   
In $\{ z \in \C: \: \im z \geq 0\}$ introduce the sets
\begin{align*}
&
Z_1 
\ :=\  
\{ z\in \C:\: \re z = 1, \ \  h^{\delta} \leq \im z \leq 1\}, \quad
0 < h \ll 1, \quad 0 < \delta < 1/2,
\cr 
&Z_2 \ :=\
 \{ z \in \C: \re z =  - 1, \ \ 0 \leq \im z \leq 1\},
\cr
&Z_3 \ :=\
 \{ z \in \C: |\re z| \leq 1,\ \  \im z = 1\}.
\end{align*}

Set $\lambda = \ii \sqrt{z}/h,\: z \in Z_1 \cup Z_2 \cup Z_3$.
To study the eigenvalues $\lambda,\: |\lambda| > R_0,$ it is sufficient to consider $0< h \ll 1$. 
 As $z$ runs over the rectangle in Figure 2, with $0< h \ll 1$,
$\lambda$ sweeps out the large values in the intersection of left and upper
half planes.  The values of $z\in Z_2$ near the lower left hand corner,
$z=-1$, of the 
 rectangle go the spectral values near the negative real axis.
The spectral analysis near these values in $Z_2$ for
dissipative Maxwell's equations does not have clear analogue with the spectral problems for the 
wave equation with dissipative boundary conditions. In fact, for the wave equation if $0 < \gamma(x) < 1,\: \forall x \in \Gamma,$ the eigenvalues of the
generator of the corresponding semigroup are located in the domain $\Lambda_{\ep}$ (see Section 3, \cite{P} and \cite{M}). For Maxwell's equations the eigenvalues of $G_b$ lie in the domain $\Lambda_{\ep} \cup {\mathcal R}_N$ and for  $0 < \gamma(x) < 1$ and $\gamma(x) > 1$ we have the same location (see Appendix for the case $K = \{x \in \R^3:\:\|x| \leq 1\}$).

\begin{center}
\begin{figure}
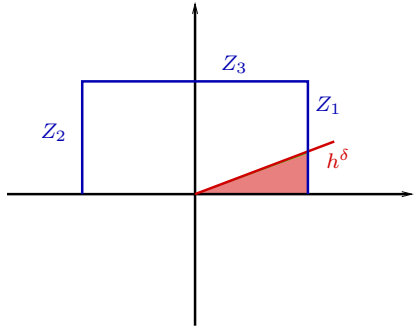
  \caption{Contours $Z_1, Z_2, Z_3, \delta = 1/2 - \epsilon$ }
\end{figure}
\end{center}

Equation (\ref{eq:1.3}) implies
 that on $\Omega$ each eigenfunction $u = (E, B)$ of $G_b$ satisfies
\begin{equation} \label{eq:1.4}
\sqrt{z} \,E\  =\  \frac{h}{\ii} \curl B\,,
\qquad
\sqrt{z} \,B \ =\  -\, \frac{h}{\ii} \curl E\,,
\end{equation}
and therefore $(-h^2 \Delta - z) E = (-h^2 \Delta  - z) B = 0.$ 
For  eigenfunctions $(E, B) \neq 0$, we 
derive  a pseudodifferential system on the boundary involving $E_{tan} = E - \la E, \nu \ra \nu$ and $E_{nor} = \la E, \nu\ra$. 
A semi-classical analysis 
shows that for $z \in Z_1\cup Z_3$ this system implies that for $h$ small enough we have $E\vert_{\Gamma}  = 0$ which yields $E = B = 0$. By scaling one concludes that the eigenvalues $\lambda = \frac{ \ii \sqrt{z}}{h}$ of $G_b$ lie in the region $\Lambda_{\ep} \cup {\mathcal M}$, where
 $${\mathcal M} = \{ z \in \C: \: |\arg z - \pi| \leq\pi/4 ,\: |z| \geq R_0 > 0, \:\re z < 0\}.$$ 
The strategy for the analysis of the case $z \in Z_1 \cup Z_3$ is similar to that exploited in \cite{V} and \cite{P}.  In these papers the semi-classical Dirichlet-to-Neumann map $\nc(z, h)$ plays a crucial role and the problem is reduced to the proof  that some $h-$pseudodifferantial operators is elliptic in a suitable class. For the Maxwell system the pseudodifferential equation on the boundary is more complicated. 
Using the equation $\dive E = 0,$ 
yields a pseudodifferential system for $E_{tan}$ and $E_{nor}.$ 
We show that if $(E, B) \neq 0$ is an eigenfunction of $G_b$, then $\|E_{nor}\|_{H^1_h(\Gamma)}$ is  bounded by $Ch \|E_{tan}\|_{H^1_h(\Gamma)}$.  The term involving $E_{nor}$ then plays the role of a negligible perturbation in the pseudodifferentrial system on the boundary and this reduces the analysis to one involving only $E_{tan}.$ The system concerning $E_{tan}$ has a diagonal leading term and we may apply the same arguments as those of \cite{P} to conclude that $E_{tan} = 0$ and hence $E_{nor} = 0.$\\

The analysis of the case $z \in Z_2$ is more difficult since the principal symbol $g$ of the pseudodifferential system for $E_{tan}$ need not be elliptic at some points (see Section 3). Even where $g$ is elliptic,   if $|\im z| \leq h^{1/2}$  it is 
 difficult to estimate the norm of the  difference $Op_h(g) Op_h(g^{-1}) - I.$ 
 To show that the eigenvalues of $G_b$ lying in ${\mathcal M}$ are in fact confined 
 to the region ${\mathcal R}_N$ for every $N \in \N$, we analyze the real part of the following scalar product in $\lg$
$$
Q(E_0) 
\ :=\  
\re \la (\nc(z, h) - \sqrt{z} \gamma)E_0, E_0 \ra_{\lg}, \qquad
E_0 
\ :=\
 E\vert_{\Gamma}.
$$
 We follow  the approach in \cite{V}, \cite{P} based on a Taylor expansion of $Q(E_0)$ at $z = -1$ and the fact that for $ z= -1 $ we have $Q(E_0) = {\mathcal O}(h^N),\: \forall N \in \N.$ In the Appendix we treat the case when $K = \{x\in \R^3:\: |x|\leq 1\}$ is a ball and $\gamma$ = const. We prove that for $\gamma \equiv 1$ the operator $G$ has no eigenvalues in $\{\re z < 0\}$, while for every $\gamma \in \R^+ \setminus \{1\}$ we have infinite number of real eigenvalues.

\section{Pseudodifferential equation on the boundary}

Introduce geodesic normal  coordinates $(y_1, y') \in \R^3$ on a neighborhood of a 
point $x_0 \in \Gamma$ as follows.  For a point $x$, $y^\prime(x)$ is the closest 
point in $\Gamma$ and $y_1 = {\rm dist}\: (x, \Gamma)$.
Define $\nu(x)$ to be the unit normal in the direction of increasing $y_1$ to the 
surface $y_1={\rm constant}$ through $x$.  Thus $\nu(x)$ is an extension of
the unit normal vector to a unit vector field.  The boundary   $\Gamma$ is mapped to  $y_1 = 0$ and
$$
x \ =\  \alpha(y_1, y') \ =\  \beta(y') \ +\  y_1 \nu(y').
$$
We have
$$
\frac{\pa}{\pa x_k} 
\ =\
  \nu_k(y') \frac{\pa}{\pa y_1} + \sum_{j=2}^3 \frac{\pa y_j}{\pa x_k} \frac{\pa}{\pa y_j},
\qquad
k = 1,2,3.
$$
Moreover, 
$$
\sum_{k=1}^3 \nu_k(y') \frac{\pa y_j}{\pa x_k}(y_1, y') = \la \nu, \frac{\pa y_j}{\pa x}\ra = 0, \qquad
j =1, 2,3,\quad {\rm and}  
$$
$$
\sum_{k=1}^3 \nu_k(x) \pa_{x_k}f(x) = \pa_{y_1}( f(\alpha(y_1, y')).$$
 Since $\|\nu(x)\| = 1$, $\la \nu, \pa_{x_j} \nu\ra = 0,\:j = 1,2,3.$

 A straight forward computation yields
$$ \nu(x) \wedge \frac{h}{\ii} \curl u(x) = \ii h \pa_{\nu} u_{tan} + \Bigl( \la  D_{x_1}u, \nu \ra, \la  D_{x_2} u, \nu \ra, \la D_{x_3}u, \nu \ra \Bigr)\Big\vert_{tan}$$
$$ 
= \ii h \pa_{\nu} u_{tan} + \Bigl(\grad_h \la u, \nu \ra\Bigr)\Big\vert_{tan} -\ii h  g_0(u_{tan}),
\qquad x \in \Gamma,
$$
where 
$$
D_{x_j} = - \ii h \pa_{x_j},\qquad
j = 1, 2, 3,\quad
\grad_h f = \{D_{x_j} f\}_{j = 1,2,3}, 
$$
$$
g_0(u_{tan}) = \{\la u_{tan}, \pa_{x_j} \nu \ra\}_{j=1,2,3}.$$ 
 Setting $E_{nor} = \la E, \nu\ra,$ from (\ref{eq:1.3}) one deduces
$$\nu \wedge B = -\frac{1}{\sqrt{z}} \nu \wedge \frac{h}{\ii}\curl E = \frac{1}{\sqrt{z}} D_{\nu} E_{tan} -  \frac{1}{\sqrt{z}}\Bigl[\Bigl(\grad_h E_{nor}\Bigr)\Big\vert_{tan} -\ii h g_0(E_{tan})\Bigr],$$
where $D_{\nu} = -\ii h\pa_{\nu}$ and the boundary condition in (\ref{eq:1.3}) becomes
\begin{equation} \label{eq:2.1}
 \Bigl( D_{\nu}   - \frac{1}{\gamma}\sqrt{z}\Bigr) E_{tan} - \Bigl(\grad_h E_{nor}\Bigr)\Big\vert_{tan}+ \ii h g_0(E_{tan})= 0,\qquad
 x \in \Gamma.
\end{equation}

Next
$$
\grad_h f(x)\vert_{tan} 
\ =\
 \Big\{\sum_{j=2}^3 \frac{\pa y_j}{\pa_{x_k}} D_{y_j} f(\alpha(y_1, y'))\Big\}_{k=1,2,3}
$$
and for $u = (u_1, u_2, u_3) \in \C^3$,
$$
\frac{h}{\ii} \dive u(\alpha(y_1, y')) = \la D_{y_1} u(\alpha(y_1,y')), \nu (y')\ra + \sum_{k=1}^3 \sum_{j= 2}^3\frac{\pa y_j}{\pa_{x_k}} D_{y_j} u_k(\alpha(y_1, y'))
$$
$$= D_{y_1}\Bigl(u_{nor}(y_1, y')\Bigr) + \sum_{j=2}^3 D_{y_j}  \Bigl\la u_{tan}(\alpha(y_1, y')), \frac{\pa y_j}{\pa x}\Bigr\ra + h \la u_{tan}, Z \ra,$$
where $\la u(\alpha(y_1, y')), \nu(y') \ra: = u_{nor}(y_1, y')$ and $Z$ depends on the second derivatives of $y_j,\: j = 2.3$.  Apply the operator $D_{y_1} - \frac{\sqrt{z}}{\gamma(y')}$ to  $\dive E(\alpha(y_1, y')) = 0$ to find
$$
(D_{y_1}^2 - \frac{\sqrt{z}}{\gamma(y')}D_{y_1}) E_{nor}(y_1, y') + \sum_{j=2}^3  D_{y_j} \Bigl\la (D_{y_1} - \frac{\sqrt{z}}{\gamma(y')}) E_{tan}(\alpha(y_1, y')), \frac{\pa y_j}{\pa x}\Bigr\ra 
$$
$$
= h \la (D_{y_1} - \frac{\sqrt{z}}{\gamma})E_{tan}, Z \ra  + h \la E_{tan}, Z_1 \ra,
$$
where $\gamma(y') := \gamma(\beta(y')).$
  
Taking the trace $y_1 = 0$ and applying the boundary condition (\ref{eq:2.1}), yields
\begin{eqnarray}
\Bigl(D_{y_1}^2 + \sum_{j, \mu = 2}^3 \sum_{k=1}^3\frac{\pa y_j}{\pa x_k} \frac{\pa y_{\mu}}{\pa x_k} D^2_{y_j,y_\mu}\Bigr) E_{nor}(0, y') - \frac{\sqrt{z}}{\gamma(y')} D_{y_1} E_{nor}(0, y')\nonumber \\
 = h \Bigl\la \Bigl(\grad_h E_{nor}\Bigr)\big\vert_{tan}(0, y'), Z \Bigr\ra + h Q_1(E_{tan}(0, y')), \label{eq:2.2}
\end{eqnarray}
with
 $$\|Q_1(E_{tan}(0, y'))\|_{L^2(\R^2)} \leq C_2 \|E_{tan}(0, y')\|_{H^1_h(\R^2)}.$$
 Here $H_h^s(\Gamma),\:s \in \R,$ denotes the semi-classical Sobolev spaces with norm $\|\langle h \pa_x \rangle^s u\|_{\lg}$, $\la h\pa_x \rangle = (1 + \|h \pa_x\|^2)^{1/2}$.  In the exposition below we use the spaces $(L^2(\Gamma))^3$ and $(H^s_h(\Gamma))^3$ of vector-valued functions but we will omit this in the notations writing simply $L^2(\Gamma)$ and $H^s_h(\Gamma)$.\\

The operator $-h^2 \Delta_x- z$ in the coordinates $(y_1, y')$  has the form
$${\mathcal P}(z, h) = D_{y_1}^2 + r(y, D_{y'}) +  q_1(y, D_y) + h^2 \tilde{q}- z$$
 with $  r(y, \eta') = \la R(y)\eta', \eta'\ra, \:q_1(y,\eta) = \langle q_1(y), \eta) \rangle$.
Here 
$$R(y) = \Bigl\{\sum_{k=1}^3 \frac{\pa y_j}{\pa x_k} \frac{\pa y_{\mu}}{\pa x_k}\Bigr\}_{j, \mu = 2}^3 = \Bigl\{\Bigl\la \frac{\pa y_j}{\pa x}, \frac{\pa y_{\mu}}{\pa x} \Bigr\ra\Bigr \}_{j, \mu =2}^3$$
is a symmetric $(2 \times 2)$ matrix  and $r(0, y', \eta') = r_0(y', \eta')$, where $r_0(y', \eta')$ is the principal symbol of the Laplace-Beltrami operator $-h^2\Delta_{\Gamma}$ on $\Gamma$ equipped with the Riemannian metric induced by the Euclidean one in $\R^3.$  We have 
$$
\Bigl({\mathcal P}(z, h) E_{nor}\Bigr)(0, y') \ =\
 \la {\mathcal P}(z, h) E, \nu \ra(0, y') \ +\  h Q_2(E(0, y')),
 $$
where 
$$
\|Q_2(E(0, y'))\|_{L^2(\R^2)} \ \leq\  C_2\, \|E(0, y')\|_{H^1_h(\R^2)}.
$$
Since ${\mathcal P}(z, h) E = 0$, this lets us replace the terms with all second derivatives of $E_{nor}$ in (\ref{eq:2.4}) by $z E_{nor}(0, y')$ modulo terms having a factor  $h$ and containing first order derivatives of $E_{nor}$. This follows from the form of the matrix $R(y)$ given above. After a multiplication by $-\frac{\gamma(y')}{\sqrt{z}}$ the equation (\ref{eq:2.2}) yields 

\begin{equation} \label{eq:2.3}
 (D_{y_1} - \gamma(y') \sqrt{z})E_{nor}(0, y') = h Q_3(E(0, y')),
\end{equation}
where $Q_3(E(0, y'))$  has the same properties as $Q_2(E(0, y')).$\\

Let $\psi(x) \in C_0^{\infty}(\R^3)$ be a cut-off function with support in small neighborhood of $x_0 \in \Gamma$. Replace $E, B$ by $E_{\psi} = E \psi,\: B_{\psi} = B \psi.$ The above analysis works for $E_{\psi}$ and $B_{\psi}$ with lower order terms depending on $\psi$. We obtain
$$\la (D_{\nu} - \gamma(x) \sqrt{z})E\vert_{\Gamma} \psi(x), \nu(x) \ra \ =\  h\, Q_{3, \psi} (E\vert_{\Gamma}). $$
Taking a partition of unity in a neighborhood of $\Gamma$, yields
\begin{equation} \label{eq:2.4}
\la(D_{\nu} - \gamma(x) \sqrt{z})E\vert_{\Gamma}, \nu\ra = h Q_{4} (E\vert_{\Gamma}) ,
\qquad
\|Q_4(E\vert_{\Gamma})\|_{L^2(\Gamma)} \leq C \|E\vert_{\Gamma}\|_{H^1_h(\Gamma)}.
\end{equation}

For $ z \in Z_1 \cup Z_2 \cup Z_3$ let 
$\rho(x', \xi',z) =  \sqrt{z - r_0(x', \xi')} \in C^{\infty}(T^* \Gamma)$
 be the root of the equation 
$$\rho^2 + r_0(x', \xi') - z = 0$$
  with $ \im \rho(x', \xi', z) > 0.$ 
  For large $|\xi'|$,
$$
\rho(x', \xi', z) \ \sim\
 |\xi'|,\qquad
 \im \rho(x', \xi', z) \ \sim\ |\xi'|,
 $$
 while for bounded $|\xi'|$,
$$
\im \rho(x', \xi', z) \ \geq\  \frac{ h^{\delta}}{C}.
$$

 We recall some basic facts about $h$-pseudodifferential operators
 that the reader can find in \cite{DS}.
 Let $X$ be a $C^{\infty}$ smooth compact manifold without
boundary with dimension $d\geq 2$. Let $(x, \xi)$ be the coordinates in $T^*(X)$ and let $a(x, \xi, h) \in C^{\infty}(T^*(X)).$ 
 Given $m \in \R, \: \l \in \R, \delta > 0$ and a function $c(h) > 0$, one denotes by $S_{\delta}^{\l, m}(c(h))$ the set of symbols so that 
$$
 |\pa_{x}^{\alpha} \pa_{\xi}^{\beta} a(x, \xi, h)| \leq C_{\alpha, \beta} (c(h))^{-\l - \delta(|\alpha| + |\beta|)} (1 + |\xi|)^{m - |\beta|},\: \forall \alpha, \forall\beta,\quad
  (x, \xi) \in T^*(X).
 $$
If $c(h) = h$, we denote $S_{\delta}^{\l, m}(c(h))$ simply by $S_{\delta}^{l, m}$. 
Symbols restricted to a domain where $|\xi|\leq C$ will be denoted by $a \in S_{\delta}^l(c(h)).$ 
The $h-$pseudodifferential operator with symbol $a(x, \xi, h)$ acts by 
$$
(Op_h(a) f)(x) \ :=\
 (2 \pi h)^{-d+1}\int_{T^*X} e^{-\ii \langle x - y, \xi \rangle /h} a(x, \xi, h) f(y) dy d \xi.
$$
For matrix valued symbols we use the same definition.
This means that every element of a matrix symbol is in the class $S_{\delta}^{l, m}(c(h)).$\\
Now suppose that $a(x, \xi, h)$ satisfies the estimates
\begin{equation} \label{eq:2.5}
|\pa_x^{\alpha} a(x, \xi, h)| \leq c_0(h) h^{-|\alpha|/2}, \qquad (x, \xi) \in T^*(X)
\end{equation}
for $|\alpha| \leq d-1$, where $c_0(h) > 0$ is a parameter. Then there exists a constant $C > 0$ independent of $h$ such that
\begin{equation} \label{eq:2.6}
\|Op_h(a)\|_{L^2(X) \to L^2(X)} 
\ \leq\
 C \,c_0(h).
\end{equation}
For $0 \leq \delta < 1/2$ products of $h$-pseudodifferential operators are well behaved. If $a \in S_{\delta}^{l_1, m_1}, \: b \in S_{\delta}^{l_2, m_2}$ and $s \in \R$, then
\begin{equation} \label{eq:2.7}
\|Op_h(a) Op_h(b) - Op_h (a b) \|_{H^s(X) \to H^{s - m_1 - m_2+1}(X)} \leq C h^{-l_1 - l_2 - 2\delta +1}.
\end{equation}

Let $u \in \C^3$ be the solution of the Dirichlet problem
\begin{equation} \label{eq:2.8}
 (-h^2 \Delta - z) u = 0 \quad {\rm on}\quad \Omega,
 \qquad
u = F \quad{\rm on} \quad \Gamma. 
\end{equation}
Introduce the semi-classical Dirichlet-to-Neumann map
$$\nc(z, h): H_h^{s}(\Gamma) \ni F \longrightarrow  D_{\nu}u\vert_{\Gamma} \in H^{s-1}_h(\Gamma).$$ 

G. Vodev \cite{V} established for bounded domains $K \subset \R^d, \: d \geq 2,$ with $C^{\infty}$
boundary the following approximation of the interior Dirichlet-to-Neumann map $\nc_{int}(z, h)$ related to (\ref{eq:2.8}), where the equation $(-h^2 \Delta - z)u = 0$ is satisfied in $K$.
\begin{theorem} [\cite{V}] For every $0 < \epsilon \ll 1$ there exists $0 < h_0(\epsilon) \ll 1$ such that for $z \in Z_{1, \ep}: = \{z \in Z_1, \: |\im z| \geq h^{\frac{1}{2}- \epsilon}\}$ and $0 < h \leq h_0(\epsilon)$ we have
\begin{equation} \label{eq:2.9}
\|\nc_{int}(z, h)(F) - Op_h(\rho + h b) F\|_{\h1} \leq \frac{C h}{\sqrt{|\im z|}} \|F\|_{L^2(\Gamma)},
\end{equation}
where $b \in S^0_{0, 1}(\Gamma)$ does not depend on $h$ and $z$. Moreover, $(\ref{eq:2.9})$ holds for $z \in Z_2 \cup Z_3$ with $|\im z|$ replaced by $1.$
\end{theorem}

With small modifications (\ref{eq:2.9})  holds for the Dirichlet-to-Neumann map $\nc(z, h)$ related to (\ref{eq:2.8}) (see \cite{P}). Applying (\ref{eq:2.9}) with $\nc(z, h)$ and $F = E_0 = E\vert_{\Gamma}$, we obtain
\begin{equation} \label{eq:2.10}
\Bigl\|\la \nc(z, h) E_0, \nu \ra  - \la Op_h(\rho) E_0, \nu\ra\Bigr\|_{L^2(\Gamma)} \leq \frac{C h}{\sqrt{|\im z|}} \|E_0\|_{\lg}.
\end{equation}
Therefore (\ref{eq:2.4}) yields 
\begin{equation} \label{eq:2.11}
\Bigl\| \la Op_h(\rho) - \gamma \sqrt{z})E_0, \nu\ra - h Q_4(E_0) \Bigl \|_{\lg} \leq \frac{Ch}{\sqrt{|\im z|}} \|E_0\|_{\lg}.
\end{equation}
The commutator $[Op_h(\rho), \nu(x)]$ is a pseudodifferential operator with symbol in $h^{1-\delta}S_{\delta}^{0, 0} $ and so
$$
\|[Op_h(\rho), \nu_k(x)] E_{nor}\|_{H^j_h(\Gamma)} 
\ \leq\  C_2 h^{1-\delta} \|E_{nor}\|_{H^j_h(\Gamma)}, \quad k=1, 2, 3,\quad j= 0, 1.$$
The last estimate combined with (\ref{eq:2.11}) implies
\begin{equation} \label{eq:2.12}
\Bigl\| ( Op_h(\rho) - \gamma \sqrt{z})E_{nor} - h Q_4(E_0) \Bigl \|_{\lg} \leq C_3\Bigl(\frac{h}{\sqrt{|\im z|}} + h^{1- \delta}\Bigr)\|E_0\|_{\lg}.
\end{equation}

\section{Eigenvalues-free regions}

For $z \in Z_{1, \ep}$ we have $\rho \in S_{\delta}^{0, 1}$ with $0 < \delta = 1/2 - \ep < 1/2$, while for $z \in Z_2 \cup Z_3$ we have $\rho \in S_{0}^{0, 1}$ (see \cite{V}). 
Since $\Gamma$ is connected one has either $\gamma(x)  >1$  or $0<\gamma(z) <1$.
We present the analysis in the case where $0<\gamma(x) <1, \: \forall x \in \Gamma$. 
The case $1 < \gamma(x)$ is reduced to this case at the end of the section.  
Clearly, there exists $\ep_0 > 0$ such that 
$$\ep_0 \leq \gamma(x) \leq 1 - \ep_0,\qquad \forall x \in \Gamma.$$

 Combing  (\ref{eq:2.4}) and (\ref{eq:2.9}), 
 yields
$$\|\la (Op_h(\rho) - \gamma(x) \sqrt{z})E_0, \nu(x)\ra \|_{L^2(\Gamma)}\leq C \frac{h}{\sqrt {|\im z}|} \|E_0\|_{\lg} + C_1 h\|E_0\|_{H^1_h(\Gamma)},$$ 
where for $z \in Z_2 \cup Z_3$ we can replace $|\im z|$ by 1. This estimate for $E_0$ and the estimate for the commutator $[Op_h(\rho), \nu_k(x)]$ imply
\begin{equation}
\label{eq:3.1}
\|( Op_h(\rho) - \gamma(x) \sqrt{z}) E_{nor} \|_{L^2(\Gamma)} 
\leq  \frac{C_3\,h}{\sqrt {|\im z}|} \|E_0\|_{\lg} + C_4 h^{1- \delta}\|E_0\|_{\h1}.
\end{equation}

Let $(x', \xi')$ be coordinates on $T^*(\Gamma)$. Consider the symbol 
$$
c(x', \xi', z)
 : =\
 \rho(x', \xi', z) 
 \ -\
  \gamma(x') \sqrt{z},\qquad x' \in \Gamma.
  $$  
  Following the analysis in Section 3, \cite{P}, we know that $c$ is elliptic in the case $0 < \gamma(x' ) < 1$ and if $z \in Z_1$ we have $c \in S_{\delta}^{0, 1},\: |\im z| c^{-1}  \in S_{\delta}^{0, -1}$, while if $z \in Z_2 \cup Z_3$ one gets $c \in S_0^{0, 1}, \: c^{-1} \in S_0^{0, -1}$. This implies
$$
\|Op_h(c^{-1}) Op_h(c) E_{nor}\|_{\h1} 
\ \leq\
 \frac{C}{|\im z|} \|Op_h(c) E_{nor} \|_{\lg}.
 $$
On the other hand, according  to Section 7 in \cite{DS}, the symbol of the operator $Op_h(c^{-1})Op_h(c) - I$ is given by
$$\sum_{j=1}^{N} \frac{(i h)^j}{j !} \sum_{|\alpha| = j} D_{\xi'}^{\alpha} (c^{-1})(x', \xi') D_{y'}^{\alpha}c(y', \eta')\big\vert_{x'= y', \xi'= \eta'} + \tilde{b}_{N}(x', \xi')$$
 $$:= b_N(x', \xi') + \tilde{b}_N(x', \xi'),$$
where
$$
|\pa_{x'}^{\alpha} \tilde{b}_N(x', \xi')| \ \leq\
 C_{\alpha} h^{N(1 - 2 \delta) -s_d- |\alpha|/2}.$$
 Taking into account the estimates for $c^{-1}$ and $c$, and applying (\ref{eq:2.5}), and (\ref{eq:2.6})
 yields
$$
\Bigl\|\Bigl(Op_h(c^{-1}) Op_h(c)- I\Bigr) E_{nor} \Bigr \|_{H^j_h(\Gamma)} \leq C_5 \frac{h}{|\im z|^{2}} \|E_{nor}\|_{H^j_h(\Gamma)}, \quad j = 0, 1.
$$
Repeating the argument in Section 3 in \cite{P} concerning the case $0 < \gamma(x') < 1$, for $z \in Z_1$ and $0 < \delta < 1/2$, one finds
\begin{equation}
\begin{aligned}
\label{eq:3.2}
\|E_{nor}&\|_{\h1} \leq \Bigl\|\Bigl(Op_h(c^{-1}) Op_h(c) - I\Bigr) E_{nor}\Bigr \|_{\h1} + \Bigl\|Op_h(c^{-1}) Op_h(c) E_{nor}\Bigr \|_{\h1}
\cr
& \leq C_6 h^{1-2 \delta} \|E_0\|_{\lg} 
 +
 C_5 h^{1- 2\delta} \|E_{nor}\|_{\h1} + C_7 h^{1-\delta}\|E_0\|_{H^1_h(\Gamma)}.
\end{aligned}
\end{equation}
Clearly, 
$$
\|E_0\|_{H^k_h(\Gamma)} \leq \|E_{tan}\|_{H^k_h(\Gamma)} + B_k \|E_{nor}\|_{H^k_h(\Gamma)},
\qquad k \in \N
$$
with $B_k$ independent of $h$. 
 Hence we can absorb the terms involving the norms of $E_{nor}$ in the right hand side of (\ref{eq:3.2}) choosing $h$ small enough, and we get
\begin{equation} \label{eq:3.3}
\|E_{nor}\|_{\h1} \leq C h^{1- 2\delta} \|E_{tan}\|_{\h1}.
\end{equation}
The analysis of the case $z \in Z_2 \cup Z_3$ is simpler since in the estimates above we have no
 coefficient $|\im z|^{-1}$ and we obtain the same result with a factor $h$ on the right hand side of (\ref{eq:3.3}). \\

With a similar argument it is easy to show that
\begin{equation} \label{eq:3.3bis}
\|E_{nor}\|_{\lg} \ \leq\
 C' h^{1- 2\delta} \|E_{tan}\|_{\lg}.
\end{equation}
In fact from (\ref{eq:2.12}) one obtains
$$\Bigl\|Op_h(c^{-1})\Bigl[ (\ Op_h(\rho) - \gamma \sqrt{z})E_{nor} - h Q_4(E_0)\Bigr] \Bigr\|_{\lg} \leq \frac{C_8}{|\im z|}\Bigl(\frac{h}{\sqrt{|\im z|}} + h^{1- \delta}\|E_0\|_{\lg}\Bigr)$$
and 
$$\|Op_h(c^{-1}) Q_4(E_0)\|_{\lg} \leq \frac{C_9}{|\im z|} \|E_0\|_{\lg}.$$
 Combining these estimates with the estimate of $\|Op_h(c^{-1}) Op_h(c) - I\|_{\lg \to \lg}$ 
yields (\ref{eq:3.3bis}).\\

Going back to the equation (\ref{eq:2.1}), we have
\begin{eqnarray} \label{eq:3.4}
\Bigl( D_{\nu} - \frac{1}{\gamma}\sqrt{z} \Bigr) E = \Bigl( D_{\nu} - \gamma\sqrt{z}\Bigr)E_{nor}\nu -  (\frac{1}{\gamma} -\gamma)\sqrt{z} E_{nor}\nu\nonumber \\
+\ii h g_0(E_{tan}) + \Bigl(\grad_h (E_{nor})\Bigr)\big\vert_{\tan},\qquad x \in \Gamma.
\end{eqnarray}
Notice that for the first term on the right hand side of (\ref{eq:3.4}) we can apply the equality (\ref{eq:2.4}), while for $ E_{nor}$ and $ \Bigl(\grad_h (E_{nor})\Bigr)\big\vert_{\tan}$ we have a control by 
the estimate (\ref{eq:3.3}). Consequently, setting $E_0 = E\vert_{\Gamma}$, the right hand side of (\ref{eq:3.4}) is bounded by $C h^{1- 2\delta} \|E_0\|_{\h1}.$
Next 
$$
1 \ < \ 
\frac{1}{1-\ep_0}
\  \leq\ 
 \frac{1}{\gamma(x)} 
 \ \leq\ 
  \frac{1}{\ep_0},\qquad \forall x \in \Gamma.
  $$ 
  This  corresponds to the case $(B)$ examined in Section 4 of
   \cite{P}. The approximation of the operator $\nc(z, h)$ given by (\ref{eq:2.9}) yields the estimate
\begin{equation} \label{eq:3.5}
\|(Op_h(\rho) - \frac{1}{\gamma}\sqrt{z})E_0\|_{\lg}\  \leq\  C \Bigl(\frac{h}{\sqrt{|\im z|}} \|E_0\|_{\lg} + h^{1- 2\delta} \|E_0\|_{\h1}\Bigr).
\end{equation}
For $z \in Z_1 \cup Z_3$  the symbol 
$$d(x', \xi', z): = \rho(x', \xi', z) - \frac{1}{\gamma(x')} \sqrt{z}$$
 is elliptic (see Section 4, \cite{P}) and $d \in S_{\delta}^{0, 1},\: d^{-1} \in S_{\delta}^{0, -1}.$ Then from  (\ref{eq:3.5}) we estimate $\|E_0\|_{\h1}$ and we obtain $E_0 = 0$ for $h$ small enough. This implies $E = B = 0$.\\

Now recall that we have
$$\re \lambda = - \frac{\im \sqrt{z}}{h}, \: \im \lambda = \frac{\re \sqrt{z}}{h}.$$
Suppose that $z \in Z_1.$ Then

$$|\re \lambda| \geq C (h^{-1})^{1- \delta},\: |\im \lambda| \leq C_1 h^{-1} \leq C_2 |\re \lambda|^{\frac{1}{1-\delta}}.$$
So if 
$$|\re \lambda| \geq C_3 |\im \lambda|^{1 - \delta}, \: \re \lambda \leq -C_4 < 0,$$
 there are no eigenvalues $\lambda = \frac{\ii \sqrt{z}}{h}$ of $G_b$. In the same way we handle the case $z \in Z_3$ and we conclude that if $z \in Z_1 \cup Z_3$ for every $\ep > 0$ the eigenvalues $\lambda = \frac{\ii \sqrt{z}}{h}$ of $G_b$ lie in the domain $\Lambda_{\ep} \cup {\mathcal M},$ where
 
$${\mathcal M}= \{ z \in \C: \: |\arg z - \pi| \leq\pi/4 ,\: |z| \geq R_0 > 0, \:\re z < 0\},$$
$\Lambda_{\ep}$ being the domain introduced in Theorem 1.1. Of course, if we consider the domain 
$$Z_{3, \delta_0} = \{ z \in \C:\: |\re z| \leq 1,  \ \ \im z = \delta_0 > 0\},$$
instead of $Z_3,$ we obtain an eigenvalue-free region with ${\mathcal M}$ replaced by 
$${\mathcal M}_{\delta_0} = \{ z \in \C: \: |\arg z - \pi| \leq {\text arctg}\: \delta_0 ,\: |z| \geq R_0(\delta_0) > 0, \:\re z < 0\}.$$

The investigation of the case $z \in Z_2$ is more  complicated since the symbol $d$ may vanish for $\im z = 0$ and $(x'_0, \xi'_0) \in T^*(\Gamma)$ satisfying the equation
$$\sqrt{1 + r_0(x'_0, \xi'_0)} - \frac{1}{\gamma(x_0')} \ =\ 0.
$$

To cover this case and to prove that the eigenvalues $\lambda = \frac{\ii \sqrt{z}}{h}$ with $z \in Z_2$ are confined in the domain ${\mathcal R}_N, \:\forall N \in \N$, we follow the arguments in \cite{V} and \cite{P}. For $z \in Z_2$ we 
 introduce an operator $T(z, h)$ that yields  a better approximation of $\nc(z, h)$. In fact, $T(z, h)$ is defined by the construction of the semi-classical parametrix in Section 3, \cite{V} for the problem (\ref{eq:2.8}) with $F = E_0$. We refer to \cite{V} for the precise definition of $T(z, h)$ and more details. For our exposition we need the next proposition.  Since $(\Delta - z)E= 0$, as in \cite{V}, we obtain 

\begin{proposition} For $z \in Z_2$ and every $N \in \N$ we have the estimate
\begin{equation} \label{eq:3.6}
\|\nc(z, h)E_0 - T (z, h) E_0\|_{H_h^1(\Gamma)} \leq C_N h^{-s_0} h^N \|E_0\|_{\lg} 
\end{equation}
with constants $C_N, s_0 > 0$, independent of $E_0, h$ and $z$, and $s_0$ independent of $N$.
\end{proposition}

\begin{proof}[Proof of Theorem 1.1 in the case $z \in Z_2$]
Consider the system
\begin{equation}\label{eq:3.7}
\begin{cases} 
 \Bigl(D_{\nu}   - \frac{1}{\gamma}\sqrt{z}\Bigr) E_{tan} - \Bigl(\grad_h E_{nor}\Bigr)\Big\vert_{tan}+ \ii h g_0(E_{tan})= 0,\quad x \in \Gamma,\\
\dive_h E_{tan} + \dive_h \Bigl(E_{nor} \nu\Bigr) = 0,\quad x \in \Gamma,
\end{cases}
\end{equation}
where $\dive_h F = \sum_{k = 1}^3 D_{x_k} F_k.$

Take the scalar product $\la , \ra_{\lg}$ in $L^2(\Gamma)$ of the first equation of (\ref{eq:3.7}) and $E_{tan}$. Applying Green formula, it easy to see that
\begin{equation} \label{eq:3.8}
- \re \la \grad_h E_{nor}\Bigl\vert_{tan}, E_{tan} \ra_{\lg} = - \re \la \dive_h E_{tan}, E_{nor}\ra_{\lg}.
\end{equation}
We claim that 
\begin{equation} \label{eq:3.9}
\im \la g_0(E_{tan}), E_{tan}\ra_{\lg} = 0.
\end{equation}
Let $E_{tan}= (w_1, w_2, w_3).$ Then
$$\la g_0(E_{tan}), E_{tan}\ra_{\C^3} = \sum_{k, j = 1}^3 w_k \frac{\pa \nu_k}{\pa x_j } \overline{w_j} =  \frac{1}{q}\sum_{k, j = 1}^3 w_k \frac{\pa V_k}{\pa x_j} \overline{w_j} =\frac{1}{q} \la S w, w \ra_{\C^3},$$
 where $S : = \{\frac{\pa V_k}{\pa x_j}\}_{k, j= 1}^3$ with $V(x) = q(x) \nu(x),\: q(x) > 0$ because $\sum_{k=1}^3 (\pa_{x_j}q)w_k \nu_k = 0$. Thus if the boundary is given locally by $x_3 = G(x_1, x_2)$, we choose $V(x) = (- \pa_{x_1}G, -\pa_{x_2} G, 1)$ and it is obvious that $S$ is symmetric. Therefore $\im \la S w, w \ra_{\C^3} = 0$ and this proves the claim.  Hence (\ref{eq:3.9}) implies
\begin{equation} \label{eq:3.10}
\re [ \ii h \la g_0(E_{tan}), E_{tan}\ra_{\lg}] = 0.
\end{equation}
From the $L^2(\Gamma)$ scalar product of the second equation in (\ref{eq:3.7}) with $E_{nor}$, we obtain
\begin{equation} \label{eq:3.11}
 \re \la \dive_h E_{tan}, E_{nor}\ra_{\lg} +  \re  \la D_{\nu} E_{nor}, E_{nor}\ra_{\lg} = 0.
\end{equation}
In fact,
$$\dive_h (E_{nor} \nu) = D_{\nu} E_{nor} - \ii h E_{nor} \dive \nu$$
and $\im \Bigl(\dive \nu |E_{nor}|^2\Bigr) = 0.$

Taking together (\ref{eq:3.8}), (\ref{eq:3.10}) and (\ref{eq:3.11}), we conclude that 
$$\re \Bigl[ \la (D_{\nu} - \frac{\sqrt{z}}{\gamma}) E_{tan}, E_{tan} \ra_{\lg} + \la D_{\nu}E_{nor}\nu, E_{nor}\nu\ra_{\lg}\Bigr] $$ 
$$= \re \Bigl \la D_{\nu} E, E \Bigl \ra _{\lg} - \re \la \frac{\sqrt{z}}{\gamma} E_{tan}, E_{tan} \ra_{\lg} = 0.$$ 

Here we have used the fact that 
$$\la D_{\nu} E_{tan}, E_{nor} \nu\ra_{\C^3} = D_{\nu}\Bigl( \la E_{tan}, E_{nor} \nu\ra_{\C^3}\Bigr) = 0.$$
Applying Proposition 3.1 with $E\vert_{\Gamma} = E_0$, yields
\begin{equation} \label{eq:3.12}
\Bigl|\re \Bigl\la T(z, h) E_0, E_0 \Bigr\ra_{\lg} - \re \Bigl\la \frac{\sqrt{z}}{\gamma} E_{tan}, E_{tan}\Bigr \ra_{\lg}\Bigr|\leq C_N h^{-s_0} h^N \|E_0\|_{\lg}.
\end{equation}
For $z = -1$, as in Lemma 3.9 in \cite{V} and Lemma 4.1 in \cite{P}, we have 
$$|\re \la T(-1, h) E_0, E_0 \ra_{\lg} | \leq C_N h^{-s_0 + N}\|E_0\|^2_{\lg} = 0.$$ 
Consequently, by using Taylor formula for the real-valued function
$$\re \Bigl[\Bigl\la T(z, h) E_0, E_0 \Bigr\ra_{\lg} -  \Big\la \frac{\sqrt{z}}{\gamma} E_{tan}, E_{tan} \Big\ra_{\lg}\Bigr],$$
 we get for every $N \in \N$ the estimate
\begin{eqnarray}
 \label{eq:3.13}
 \Bigl|\im \Bigl[
 \Big\la (\frac{\pa T}{\pa z}(z_t, h))E_0, E_0 
 \Big\ra_{\lg} - 
 \Big\la\frac{\gamma_1}{2\sqrt{z_t}} E_{tan}, E_{tan}
 \Big\ra_{\lg} \Bigr]\Bigr|\nonumber \\
\leq C_N \frac{ h^{-s_0 + N}}{|\im z|} \|E_0\|^2_{\lg},
\end{eqnarray}
where $z_t = -1 + \ii t \im z,\: 0 < t < 1, \: \gamma_1 = \gamma^{-1}.$

According to  Lemma 3.9 in \cite{V}, in (\ref{eq:3.13}) we can replace $\frac{\pa T}{\pa z}(z_t, h)$ by $Op_h(\frac{\pa \rho}{\pa z}(z_t))$ and this yields an error term bounded by $C h \|E_0\|^2_{H^{-1}_h(\Gamma)}$. On the other hand,
$$\Bigl|\Bigl\la Op_h( \frac{\pa \rho}{\pa z}(z_t))E_{tan}, E_{nor}\nu\Bigr\ra _{\lg} + \Bigl\la Op_h( \frac{\pa \rho}{\pa z}(z_t))E_{nor}, E_{tan}\nu\Bigr\ra _{\lg}\Bigr|$$
$$\leq C h \|E_0\|_{\lg}^2$$
since the estimate (3.4) holds for $z \in Z_2$ with factor $h$ and $\frac{\pa \rho}{\pa z}(z_t) \in S^{0, -1}_0.$

Thus  the problem is reduced to a lower bound of 
$$J : = \Bigl|\im \Bigl[\la \Bigl(Op_h( \frac{\pa \rho}{\pa z}(z_t))- \frac{\gamma_1}{2 \sqrt{z}}\Bigr) E_{tan}, E_{tan} \ra_{\lg} + \la Op_h( \frac{\pa \rho}{\pa z}(z_t))E_{nor} \nu, E_{nor}\nu\ra_{\lg}\Bigr]\Bigr| $$
$$ \geq \Bigl|\im\la \Bigl(Op_h( \frac{\pa \rho}{\pa z}(z_t))- \frac{\gamma_1}{2 \sqrt{z}}\Bigr) E_{tan}, E_{tan} \ra_{\lg}\Bigr| - C_1 \|E_{nor}\|^2_{\lg}.$$
Since $\gamma_1(x) > 1, \: \forall x\in \Gamma$,  applying the analysis of Section 4 in \cite{P} for the scalar product involving $E_{tan}$, one deduces 
$$\Bigl|\im \Bigl\la \Bigl(Op_h( \frac{\pa \rho}{\pa z}(z_t))- \frac{\gamma_1}{2 \sqrt{z}}\Bigr) E_{tan}, E_{tan}\Bigr \ra_{\lg}\Bigr| \geq \eta_1 \|E_{tan}\|^2_{\lg},\quad \eta_1 > 0.$$
By using once more the estimate (\ref{eq:3.3bis}), for $h$ small enough we obtain
$$J \geq \eta_1\Bigl (\|E_{tan}\|^2_{\lg} + \|E_{nor}\|^2_{\lg}\Bigr) - B_0 h \|E_{tan}\|^2_{\lg} \geq \eta_2 \|E_0\|^2_{\lg}, \ \ 0 < \eta_2 < \eta_1.$$
Consequently,  (\ref{eq:3.13}) yields
$$(\eta_2 - B_1 h)\|E_0\|^2_{\lg}  \leq C_N \frac{ h^{-s_0 + N}}{|\im z|} \|E_0\|^2_{\lg}$$
and for small $h$ we conclude that for $z \in Z_2$ the eigenvalues $\lambda = \frac{\ii \sqrt{z}}{h}$ of $G_b$ lie in the region ${\mathcal R}_{N}.$ This completes the analysis of the case $0 < \gamma(x) < 1,\: \forall x \in \Gamma.$\\

To study the case $\gamma(x) > 1, \: \forall x \in \Gamma,$ we write the boundary condition in (\ref{eq:1.1}) as
$$\frac{1}{\gamma(x)} (\nu \wedge E_{tan})  - (\nu \wedge (\nu \wedge B_{tan})) =  \frac{1}{\gamma(x)} (\nu \wedge E_{tan}) + B_{\tan} = 0.$$
Next
 $$\nu \wedge E = \frac{1}{\sqrt{z}} \nu \wedge \frac{h}{\ii}\curl B = -\frac{1}{\sqrt{z}} D_{\nu} B_{tan} +  \frac{1}{\sqrt{z}}\Bigl[\Bigl(\grad_h B_{nor}\Bigr)\Big\vert_{tan} -\ii h g_0(B_{tan})\Bigr]$$
and one obtains
\begin{equation} \label{eq:3.15}
 \Bigl( D_{\nu}   - \gamma(x)\sqrt{z}\Bigr) B_{tan} - \Bigl(\grad_h B_{nor}\Bigr)\Big\vert_{tan}+ \ii h g_0(B_{tan})= 0,\: x \in \Gamma
\end{equation}
which is the same as  (\ref{eq:2.1}) with $E_{tan}$, $E_{nor}$ replaced respectively by $B_{tan},\: B_{nor}$ and $\frac{1}{\gamma(x)}$ replaced by $\gamma(x) > 1.$  We apply the operator $D_{y_1} - \gamma \sqrt{z}$ to the equation $\dive B = 0$ and repeat without any change the above analysis concerning $E_{tan}, \: E_{nor}.$ Thus the proof of Theorem 1.1 is complete. \end{proof}

\begin{remark}
The result of Theorem 1.1 holds for obstacles $K = \cup_{j= 1}^J K_j$, where $K_j,\: j = 1,...,J$ are open connected domains with $C^{\infty}$ boundary and $K_i \cap K_j = \emptyset,\: i \neq j.$ Let $\Gamma_j = \pa K_j, \: j = 1,...,J.$ In this case we may have $\gamma(x) < 1$ for some obstacles $\Gamma_j$ and $\gamma(x) > 1$ for other ones. The proof extends with only minor modifications.  The construction of the semi-classical parametrix in \cite{V} is local 
and for the Dirichlet-to-Neumann map $\nc_j(z, h)$ related to $\Gamma_j$ we get the estimate
$$\|\nc_j(z, h)(F) - Op_h(\rho + h b) F\|_{H^1_h(\Gamma_j)} \leq \frac{C h}{\sqrt{|\im z|}} \|F\|_{L^2(\Gamma_j)}.$$
The boundary condition in (\ref{eq:1.1}) is local and we can reduce the analysis to a fixed obstacle $K_j$. If $(E, B) \neq 0$ is an eigenfunction of $G_b$, our argument implies $E_{tan} = 0$ for  $x \in \Gamma_j$ if $0 < \gamma(x) < 1$ on $\Gamma_j$ and $B_{tan} = 0$ for  $x \in \Gamma_j$ in the case $\gamma(x) > 1$ on $\Gamma_j.$ By the boundary condition we get $E_{tan} = 0$ on $\Gamma$ and this yields $E=B =0$ since the Maxwell system with boundary condition
$E_{tan} = 0$ has no eigenvalues in $\{z\in \C:\: \Re z < 0\}.$
\end{remark}

\section{Appendix}
\def\hn{h^{(1)}_n}

In this Appendix, assume that $\gamma > 0$ is constant. 
Our purpose is to study  the  eigenvalues of $G_b$ in case the obstacle is equal to the 
 ball $B_3 = \{x \in \R^3: |x| \leq 1\}$. Setting $\lambda = \ii \mu$, $\im \mu > 0,$ 
 an eigenfunction $(E, B) \neq 0$ of $G_b$ 
 satisfies
\begin{equation} \label{eq:4.1}
\curl E = - \ii \mu B,\qquad  \curl B = \ii \mu E.
\end{equation}
Replacing $B$ by $H = - B$ yields for $(E, H) \in (H^2(|x| \leq 1))^6$,
\begin{equation} \label{eq:4.2}
\begin{cases}\curl E =  \ii \mu H,\qquad  \curl H = -\ii \mu E,\quad {\rm for}\quad x\in B_3,  \\
E_{tan} + \gamma (\nu \wedge H_{tan})= 0,\quad {\rm for }\quad x \in \S^2.
\end{cases}
\end{equation}
Expand $E(x), H(x)$ in
 the spherical functions $Y_n^m(\omega), \: n = 0, 1, 2,..., \: |m| \leq n,\: \omega \in \S^2$ and the modified
Hankel functions $\hn(z)$ of first kind. 
An application of Theorem 2.50 in \cite{KH} (in the notation of \cite{KH} it is necessary to replace $\omega$ by $\mu \in \C \setminus \{0\}$) says that the solution of the system (\ref{eq:4.2}) for $|x| = r = 1$ has the form
$$E_{tan}(\omega) = \sum_{n=1}^{\infty}\sum_{|m|\leq n} \Bigl[\alpha_n^m \Bigl(\hn(\mu) + \frac{d}{dr}\hn(\mu r)\vert_{r = 1}\Bigr) U_n^m (\omega) + \beta_n^m \hn(\mu) V_n^m(\omega)\Bigr],$$
 
$$H_{tan}(\omega) = -\frac{1}{\ii \mu} \sum_{n=1}^{\infty}\sum_{|m|\leq n} \Bigl[\beta_n^m \Bigl(\hn(\mu) + \frac{d}{dr}\hn(\mu r)\vert_{r = 1}\Bigr) U_n^m (\omega) + \mu^2\alpha_n^m \hn(\mu) V_n^m(\omega)\Bigr].$$

Here $U_n^m (\omega) = \frac{1}{\sqrt{n(n+1)}} \grad_{\S^2} Y_n^m(\omega)$ and $V_n^m(\omega) = \nu \wedge U_n^m(\omega)$ for $ n \in \N, -n \leq m \leq n$ form a complete orthonormal basis in
$$L^2_t(\S^2) = \{ u \in (L^2(\S^2))^3:\: \la \nu, u \ra = 0\: {\rm on}\: S^2\}.$$
To find a representation of $\nu \wedge H_{tan}$,  
observe that $\nu \wedge (\nu \wedge U_n^m) = - U_n^m,$ so
$$(\nu \wedge H_{tan})(\omega) =  -\frac{1}{\ii \mu} \sum_{n=1}^{\infty}\sum_{|m|\leq n} \Bigl[\beta_n^m \Bigl(\hn(\mu) + \frac{d}{dr}\hn(\mu r)\vert_{r = 1}\Bigr) V_n^m (\omega) - \mu^2\alpha_n^m \hn(\mu) U_n^m(\omega)\Bigr]$$
and the boundary condition in (\ref{eq:4.2}) is satisfied if
\begin{equation} \label{eq:4.3}
\alpha_n^m \Bigl[ \hn(\mu) + \frac{d}{dr}(\hn(\mu r))\vert_{r = 1} - \gamma \ii \mu \hn(\mu)\Bigr] = 0,\: \forall n \in \N, \: |m|\leq n,
\end{equation}
\begin{equation}\label{eq:4.4}
-\frac{\beta_n^m \gamma}{\ii \mu}\Bigl[  \hn(\mu) + \frac{d}{dr}(\hn(\mu r))\vert_{r = 1} - \frac{\ii \mu}{\gamma} \hn(\mu)\Bigr] = 0,\: \forall n \in \N, \: |m|\leq n.
\end{equation}
For $\gamma \equiv 1$, there are no  eigenvalues. 

\begin{proposition} For $\gamma \equiv 1$ the operator $G_b$ has no eigenvalues in $\{\re z < 0\}.$
\end{proposition}

\begin{proof} 
The functions $\hn(z)$ have the form (see for example \cite{O})
$$h^{(1)}_n(x)
\  =\
 (-\ii)^{n+1} \frac{e^{\ii x}}{x} \sum_{m=0}^n \frac{\ii^m}{m!(2x)^m} \frac{(n+m)!}{(n-m)!}
 \  =\ (-\ii)^{n+1} \frac{e^{ix}}{x} R_n\Bigl(\frac{\ii}{2 x}\Bigr)$$
with
$$R_n(z)\ : =\
 \sum_{m=0}^n \frac{z^m}{m!} \frac{(n+m)!}{(n-m)!} \ =\
  \sum_{m= 0}^n a_m z^m.$$
Therefore the term in the brackets $[...]$ in (\ref{eq:4.3}) becomes
$$
(1 - \gamma) \ii \mu R_n\Bigl(\frac{\ii}{2 \mu}\Bigr) \  -\
 \sum_{m=0}^n a_m m \Bigl(\frac{\ii}{2 \mu}\Bigr)^{m}\,.
 $$
Setting $w = \frac{\ii}{2 \mu},$ we must study for $\re w > 0$ the roots of the equation
\begin{equation} \label{eq:4.5}
g_n(w)\ : =\
 \frac{1 - \gamma}{2 w} R_n(w) + w R_n'(w) \ =\  0\,.
\end{equation}
For $\gamma = 1$ one obtains $R_n'(w) = 0.$ A result of Macdonald says that the zeros of the function $\hn(z)$ lie in the half plane $\im z < 0$ (see Theorem 8.2 in \cite{O}), hence $R_n(w) \neq 0$ for $\re w \geq 0.$ By the theorem of Gauss-Lucas we deduce that the roots of $R_n'(w) = 0$ lie in the convex hull of the set of the roots of $R_n(w) = 0,$ so $R_n'(w) \neq 0$ for $\re w > 0.$ Consequently, (\ref{eq:4.3}) and (\ref{eq:4.4}) are satisfied
only for $\alpha_n^m = \beta_n^m = 0$ and $E_{tan} = 0.$ This implies $E = H = 0.$ \end{proof}

For the case $\gamma \neq 1$, there are an infinite number of real eigenvalues.

\begin{proposition} Assume that $\gamma \in \R^+ \setminus \{1\}$ is a constant. Then $G_b$ has an infinite number of real eigenvalues. Let $\gamma_0 = \max\{ \gamma, \frac{1}{\gamma}\}.$ Then all real eigenvalues $\lambda$ with exception of the eigenvalue
\begin{equation} \label{eq:4.6}
\lambda_1 \ =\
 -\frac{2}{(\gamma_0 -1)\Big(1 + \sqrt{1 + \frac{4}{\gamma_0 - 1}}\Big)}
\,.
\end{equation}
satisfy the estimate 
\begin{equation} \label{eq:4.7}
\lambda 
\ \leq\
 - \frac{1}{\max \{ (\gamma_0 - 1), \sqrt{\gamma_0 - 1}\}}
 \,.
\end{equation}
\end{proposition}

\begin{proof} Assume first that $\gamma > 1$. Then $q_n(w) =w g_n(w) = 0$ has at least one real root $w_0 > 0.$ Indeed, $q_n(0) = \frac{1- \gamma}{2} < 0,\: q_n(w) \to +\infty$ as $w \to + \infty.$ Choosing $\alpha_n^{m_0} \neq 0$ for an integer $m_0,\: |m_0| \leq n$ and taking all
other coefficients $\alpha_n^m,\: \beta_n^m$ equal to 0, 
yields $E_{tan} \neq 0$ and $G_b$ has an eigenfunction with eigenvalue $\lambda = -\frac{1}{2w_0} < 0.$

 It is not excluded that $g_n(w)$ and $g_m(w)$ for $n \neq m$ have the same real positive root. If we assume that for $\re w > 0$ the sequence of functions $\{g_n(w)\}_{n=1}^{\infty}$ has only a finite number of real roots $w_1,...,w_N$, $w_j \in \R^+$, then there exists an infinite number of functions $g_{n_j}(w)$ having  the same root which implies that we have an eigenvalue of $G_b$ with infinite multiplicity. This is a contradiction, and the number of real eigenvalues of $G_b$ is infinite. 

 It remains to establish the bound on the real eigenvalues. First, consider the case $n = 1$. Then one obtains the equation
$$\frac{2 w^2}{2 w + 1} = \frac{\gamma -1}{2}$$
which has a positive root $w_0 = \frac{1}{4} \Bigl(\gamma - 1 + \sqrt{(\gamma - 1)^2 + 4 (\gamma - 1)}\Bigr)$. This yields the  $\lambda_1$ from (\ref{eq:4.6})

Next  examine the case $n \geq 2.$  For a root  $w_0 \in \R^+$ one has
$$
w_0\bigg(
w_0 \frac{R_n'(w_0)}{R_n(w_0)} \bigg)
\ =\ 
 \frac{\gamma - 1}{2}\,.
 $$

{\bf Case 1.}  $w_0 \geq \frac{1}{2\sqrt{3}}.$ 
 Then the inequality
$$
\frac{\sum_{m = 2}^n m a_m w_0^m + a_1w_0 }{\sum_{m = 2}^n  a_m w_0^m + a_1 w_0 + 1} 
\ \geq\
 \frac{2\sum_{m = 2}^n  a_m w_0^m + a_1w_0}{\sum_{m = 2}^n  a_m w_0^m + a_1 w_0 + 1}
 \ \geq 1
 $$
is satisfied since $a_2 = \frac{1}{2} (n+2)(n+1)n(n-1) \geq 12.$ Consequently,
$ 2 w_0 \leq \gamma - 1$ and this implies that
 the eigenvalue $\lambda = -\frac{1}{2 w_0}$ 
 satisfies
\begin{equation} \label{eq:4.8}
\lambda\ <\
 - \frac{1}{\gamma - 1}\,.
\end{equation}
{\bf Case 2.} $0 < w_0 \leq \frac{1}{2 \sqrt{3}}.$
 Apply the inequality
$$
\frac{\sum_{m=2}^n m a_m w_0^{m-1} + a_1}{w_0 \sum_{m=2}^{n} a_m w_0^{m-1} + a_1 w_0+  1} 
\ \geq\ 
 \frac{2\sum_{m=2}^{n} a_m w_0^{m-1} + a_1}{w_0 \sum_{m=2}^{n} a_m w_0^{m-1} + a_1 w_0 + 1}
 \ \geq\
  2
  $$
that is equivalent to
$$
2\Bigl[(1 - w_0) S_0 -  a_1 w_0\Bigr] + a_1 \geq  2
$$
with $S_0 = \sum_{m=2}^{n} a_m w_0^{m-1}$.
This inequality holds because 
$$
(1 - w_0)\sum_{m=2}^{n} a_m w_0^{m-1} - a_1 w_0\ \geq\
 (\frac{1}{2} a_2 - a_1)w_0\,,
 \quad
 a_1 = (n+1)n \geq 2,
 $$ 
and,
$$
\frac{1}{2} a_2 - a_1 = \frac{1}{4} (n+2) (n+1) n (n-1) - (n+1) n = n(n+1) \Bigl[ \frac{1}{4}(n+2)(n-1) - 1\Bigr] \geq 0.
$$

Therefore,
$$
2w_0^2 
\ \leq\
 w_0^2 \frac{\sum_{m=1}^n m a_m w_0^{m-1}}{\sum_{m=1}^{n} a_m w_0^{m} + 1} 
 \ = \
 \frac{\gamma - 1}{2}\,.
 $$
This easily yields
\begin{equation} \label{eq:4.9}
\lambda \leq - \frac{1}{\sqrt{
(\gamma - 1)}}.
\end{equation}

 In the case $0 < \gamma < 1$ one has $1/\gamma > 1$ and we apply the above analysis to the equation (\ref{eq:4.4}). Setting $\gamma_0 = \max \{ \gamma , \frac{1}{\gamma}\}$ and taking into account (\ref{eq:4.8}) and (\ref{eq:4.9}), we obtain the result. This completes the proof.
\end{proof}

\begin{remark}
Proposition 4.2 yields a more precise result than that in \cite{CPR} since we prove the existence of an infinite number of real eigenvalues $G_b$ for every $\gamma \in \R^+ \setminus \{1\}.$ In the case $\gamma = \frac{1}{1 + \ep},\: \ep > 0$ the eigenvalue $\lambda_1$ has the form
$$\lambda_1 = \frac{1}{2}\Bigl(1 - \sqrt{1 + \frac{4}{\epsilon}}\Bigr)$$
and this result for small $\epsilon >0$ has been obtained in \cite{CPR}.
Clearly, as $\gamma \to 1$ the real eigenvalues of $G_b$ go to $-\infty.$ . 
\end{remark} 

It is easy to see that for $\gamma > 1$ the equation $g_n(w) = 0$ has no complex roots. 
Denote by
$$z_j, \quad \re z_j < 0,\quad j = 1,...,n,\: n \geq 1$$
 the roots of $R_n(w) = 0$.  Suppose that $ g_n(w_0) = 0,\: n \geq 1$ with $\re w_0 > 0, \: \im w_0 \neq 0.$ Then
$$
\im \Bigl[\frac{1 - \gamma}{2 w_0} + w_0 \sum_{j=1}^n \frac{1}{w_0 - z_j}\Bigr] = 0
$$
 and

\begin{align}
\label{eq:4.10}
-\frac{(1- \gamma) \im w_0}{2 |w_0|^2} + \re w_0\bigg[- \sum_{j=1}^n & \frac{\im w_0}{|w_0- z_j|^2} + \sum_{j =1}^n \frac{ \im z_j}{|w_0- z_j|^2}\bigg]
\cr
 &+ \im w_0 \sum_{j=1}^n \frac{\re w_0 - \re z_j}{|w_0- z_j|^2} \ =\  0\,.
\end{align}

 On the other hand, if $z_j$ with $\im z_j \neq 0$ is a root of $R_n(w)=0$, then $\bar{z}_j$ is also a root and
\begin{align*}
\frac{\im z_j}{|w_0 - z_j|^2} - \frac{\im z_j}{|w_0 - \bar{z}_j|^2}
&= \frac{\im z_j}{|w_0 - z_j|^2 |w_0 - \bar{z}_j|^2} \Bigl( |w_0- \bar{z}_j|^2 - |w_0 - z_j|^2\Bigr)
\cr
&= \frac{4 \im w_0 (\im z_j)^2}{|w_0 - z_j|^2 |w_0 - \bar{z}_j|^2}\ .
\end{align*}

Equation (\ref{eq:4.10}) becomes
\be \label{eq:4.11}
\im w_0 \bigg[
 \frac{\gamma - 1}{2|w_0|^2} - \sum_{j=1}^n \frac{\re z_j}{|w_0- z_j|^2} + \sum_{\im z_j > 0}\frac{4 \re w_0 (\im z_j)^2}{|w_0- z_j|^2 |w_0- \bar{z}_j|^2}
 \bigg] 
\ =\ 
 0\,.
\ee
 The term in the brackets $[...]$ is positive, and one concludes that $\im w_0 = 0.$ 

Repeating the argument of the Appendix in \cite{P}, 
 one can show that {\it for $0 < \gamma < 1$ the complex eigenvalues 
of $G_b$ lie in the region
$$
\Big\{z \in \C: \: |{\rm arg}\:z - \pi| > \pi/4, \quad \re z < 0\Big\}.$$
}

\begin{remark}
We do not know if there exist   non real eigenvalues for $B_3$.
\end{remark}

\vskip.3cm

{\bf  Acknowledgment.} We thank
 Georgi Vodev  for many useful discussions and remarks concerning 
 an earlier version of the paper.

\end{document}

%% file: fig1.pdf_t
\begin{picture}(0,0)%
\includegraphics{fig1.pdf}%
\end{picture}%
\setlength{\unitlength}{2072sp}%
\begingroup\makeatletter\ifx\SetFigFont\undefined%
\gdef\SetFigFont#1#2#3#4#5{%
  \reset@font\fontsize{#1}{#2pt}%
  \fontfamily{#3}\fontseries{#4}\fontshape{#5}%
  \selectfont}%
\fi\endgroup%
\begin{picture}(4926,3936)(2218,-5269)
\put(6076,-3346){\makebox(0,0)[lb]{\smash{{\SetFigFont{8}{9.6}{\rmdefault}{\mddefault}{\updefault}{\color[rgb]{.82,0,0}$h^{\delta}$}%
}}}}
\put(5941,-2671){\makebox(0,0)[lb]{\smash{{\SetFigFont{8}{9.6}{\rmdefault}{\mddefault}{\updefault}{\color[rgb]{0,0,.69}$Z_{1}$}%
}}}}
\put(2656,-2986){\makebox(0,0)[lb]{\smash{{\SetFigFont{8}{9.6}{\rmdefault}{\mddefault}{\updefault}{\color[rgb]{0,0,.69}$Z_{2}$}%
}}}}
\put(4816,-2176){\makebox(0,0)[lb]{\smash{{\SetFigFont{8}{9.6}{\rmdefault}{\mddefault}{\updefault}{\color[rgb]{0,0,.69}$Z_{3}$}%
}}}}
\end{picture}%